\documentclass[11pt,twoside]{article}
\usepackage[T1]{fontenc}
\usepackage{textcomp}
\usepackage{mathrsfs}
\usepackage{amsmath}
\usepackage{amssymb}
\usepackage{fancyhdr}
\usepackage{latexsym}
\usepackage{bbding}
\usepackage{mathrsfs}
\usepackage{wasysym}
\usepackage{cite}
\usepackage{multicol,graphics}
\usepackage{xcolor}

\RequirePackage{times}

\def\XXint#1#2#3{{\setbox0=\hbox{$#1{#2#3}{\int}$ }
		\vcenter{\hbox{$#2#3$ }}\kern-.6\wd0}}

\newtheorem{theorem}{Theorem}[section]
\newtheorem{lemma}[theorem]{Lemma}

\newtheorem{definition}[theorem]{Definition}

\newtheorem{proposition}[theorem]{Proposition}
\numberwithin{equation}{section}
\newenvironment{proof}[1][Proof]{\noindent\textbf{#1.} }{\hfill $\Box$}

\allowdisplaybreaks[4]
\makeatletter
\begin{document}
	
\title{On the Cauchy problem for the fractional Keller--Segel system in variable Lebesgue spaces\footnote{E-mail: gaston.vergarahermosilla@univ-evry.fr, jihzhao@163.com.}}
\author{Gast\'on Vergara-Hermosilla$^{\text{1}}$, Jihong Zhao$^{\text{2}}$\\
[0.2cm] {\small $^{\text{1}}$LaMME, Univ. Evry, CNRS, Universit\'e Paris-Saclay, 91025, Evry, France}\\
[0.2cm] {\small $^{\text{2}}$School of Mathematics and Information Science, Baoji University of Arts and Sciences,}\\
[0.2cm] {\small  Baoji, Shaanxi 721013,  China}}

\date{\today}
\maketitle
\begin{abstract}
In this paper, we are mainly concerned with the well-posed problem of the fractional Keller--Segel system in the framework of  variable Lebesgue spaces. Based on carefully examining the algebraical structure of the system, we reduced the fractional Keller--Segel system into the generalized nonlinear heat equation to overcome the difficulties caused by the boundedness of the Riesz potential in a variable Lebesgue spaces, then by mixing some structural properties of the variable Lebesgue spaces with the optimal decay estimates of the fractional heat kernel, we were able to establish two well-posedness results of the fractional Keller--Segel system in this functional setting.
\end{abstract}
\smallbreak

\textit{Keywords}: Fractional Keller--Segel system; well-posedness; variable Lebesgue spaces.
\smallskip

\textit{2020 AMS Subject Classification}:   35A01,  35M11, 46E30, 92C17

\section{Introduction}
In this paper we study the Cauchy problem of the fractional Keller--Segel system defined in the whole space $\mathbb{R}^n$:
\begin{equation}\label{eq1.1}
\begin{cases}
  \partial_{t}u+\Lambda^{\alpha}u+\nabla\cdot(u\nabla \phi)=0,\\
    -\Delta \phi=u,\\
  u(0,x)=u_0(x),\qquad x\in \mathbb{R}^n
\end{cases}
\end{equation}
where $n\geq2$, $u$ and $\phi$ are two unknown functions which stand for the cell density and the concentration of the chemical attractant, respectively, and the anomalous (normal) diffusion is modeled by a fractional power of the Laplacian with  $1<\alpha\leq2$. The positive operator $\Lambda^{\alpha}=(-\Delta)^{\frac{\alpha}{2}}$ is defined by
\begin{equation*}
  \Lambda^{\alpha}u(x):
    =\sigma_{n,\alpha}P.V.\int_{\mathbb{R}^{n}}\frac{u(x)-u(y)}{|x-y|^{n+\alpha}}dy
\end{equation*}
and $\sigma_{n,\alpha}$ is a normalization constant. A simple alternative representation
is given through the Fourier transform as
$\Lambda^{\alpha}u=\mathcal{F}^{-1}[|\xi|^{\alpha}\mathcal{F}u(\xi)]$,
where
$\mathcal{F}$ and $\mathcal{F}^{-1}$ are the Fourier transform and the inverse Fourier transform, respectively.
\smallbreak

 The system \eqref{eq1.1} was first proposed by Escudero in \cite{E06},  where it was used to describe the spatiotemporal distribution of a population density of random walkers undergoing L\'{e}vy flights.  Obviously, when $\alpha=2$, the system \eqref{eq1.1} becomes to the well-known elliptic-parabolic Keller--Segel system:
\begin{equation}\label{eq1.2}
\begin{cases}
  \partial_{t}u-\Delta u=-\nabla\cdot(u\nabla \phi),\\
  -\Delta \phi=u,\\
  u(0,x)=u_0(x),\qquad \mathbb{R}^n.
\end{cases}
\end{equation}
This system is  a mathematical model of chemotaxis, which was formulated  by Keller--Segel \cite{KS70} in 1970s, while it is also connected with astrophysical models of
gravitational self-interaction of massive particles in a cloud or a nebula (cf. \cite{BCGK04, BHN94B}). It is well-known that the system \eqref{eq1.2}  admits global solution if the initial total mass $\int_{\mathbb{R}^{n}}u_{0}(x)dx$ ($n\ge2$) is small enough, and it may blow up in finite time for large initial data, we refer the readers to see \cite{H03, H04} for a comprehensive review of these aspects of parabolic-elliptic system like \eqref{eq1.2}, and to \cite{I11, L13, BKZ15, LW23,  LYZ23, NY20, NY22,  XF22} for very recent well-posedness and ill-posedness results on the related system.
\smallbreak

For the fractional Keller--Segel system \eqref{eq1.1},  We know from \cite{E06, OY01} that the one dimensional system \eqref{eq1.1} possesses global solutions not
only in the case $\alpha=2$ but also in the case $1<\alpha<2$. However, a singularity appears in finite time when $0<\alpha<1$ if the
initial configuration of cells is sufficiently concentrated, see the paper \cite{BC10} due to Bournaveas--Calvez  for a extended discussion about it. Moreover, when $n\ge2$, the solutions of \eqref{eq1.1} globally exist for small initial data and may blow up in finite time for large initial data. In the following  we list some analytical results concerning about the global existence or blow-up for this system:
\begin{itemize}
\item For $1<\alpha<2$, Biler--Karch \cite{BK10}  established global well-posedness of system \eqref{eq1.1} with small initial data in the critical Lebesgue space $L^{\frac{n}{\alpha}}(\mathbb{R}^{n})$, they also proved the finite time blowup of nonnegative solutions with initial data imposed on large mass or high concentration conditions;
 \item For $1<\alpha<2$, Biler--Wu \cite{BW09} established  global well-posedness of system \eqref{eq1.1} with small initial data in  critical  Besov space $\dot{B}^{1-\alpha}_{2,q}(\mathbb{R}^{2})$;
\item For $1<\alpha\leq2$, Zhai \cite{Z10} proved global existence, uniqueness and stability of solutions of system \eqref{eq1.1} in critical Besov space with general potential type nonlinear term;
\item For $1<\alpha \leq 2$, Wu--Zheng \cite{WZ11} proved local well-posedness and global well-posedness of system \eqref{eq1.1} with small initial data in critical Fourier--Herz space $\mathcal{\dot{B}}^{2-2\alpha}_{q}(\mathbb{R}^{n})$ for $1<\alpha\leq2$ and $2\leq q\leq \infty$, they also proved  that system \eqref{eq1.1} is  ill-posedness in  $\mathcal{\dot{B}}^{-2}_{q}(\mathbb{R}^{n})$ and  $\dot{B}^{-2}_{\infty, q}(\mathbb{R}^{n})$ with $\alpha=2$ and $2<q\leq\infty$;
\item For $0<\alpha\leq 1$, Sugiyama--Yamamoto--Kato \cite{SYK15} proved local existence with large data and global existence with small data in critical Besov space $\dot{B}^{-\alpha+\frac{n}{p}}_{p,q}(\mathbb{R}^{n})$ for $n\geq 3$, $2\leq p<n$ and  $1\leq q<2$;
 \item For $1<\alpha\leq 2$, Zhao \cite{Z18} showed global existence and analyticity of solutions with small initial data in critical Besov space $\dot{B}^{-\alpha+\frac{n}{p}}_{p,q}(\mathbb{R}^{n})$ ($1\leq p<\infty$,  $1\leq q\leq \infty$) and $\dot{B}^{-\alpha}_{\infty,1}(\mathbb{R}^{n})$. For the limit case $\alpha=1$, the global existence and analyticity of solutions with small initial data in critical Besov space $\dot{B}^{-1+\frac{n}{p}}_{p,1}(\mathbb{R}^{n})$ ($1\leq p<\infty$) and $\dot{B}^{-1}_{\infty,1}(\mathbb{R}^{n})$ were also obtained;
 \item  Parts of these results were also extended for the other generalized chemotaxis models and the fractional power drift-diffusion system of bipolar type,
 we refer the readers to see \cite{CLW19, DL17, SYK15, YKS14, Z21} for more results.
 \end{itemize}
\smallbreak

Motivated by a series of papers \cite{CV24, V241, V242}, in this paper we study questions concerning about the well-posedness of system \eqref{eq1.1} in the framework of the Lebesgue spaces with variable exponent. The variable Lebesgue spaces are a generalization of the classical Lebesgue spaces, replacing the constant exponent $p\in[1,+\infty)$ with a variable exponent function $p(\cdot):\mathbb{R}^{n}\rightarrow[1,+\infty)$. As we will see in the following lines,  the resulting Banach spaces $L^{p(\cdot)}(\mathbb{R}^n)$ have many properties similar to the $L^{p}(\mathbb{R}^n)$ spaces, however, such spaces allow some functions with singularity at some point $x_{0}$, which is not in $L^{p}$ for any $1\leq p<\infty$ due to it either grows too quickly at $x_{0}$ or decays too slowly at infinity, thus there are surprising and subtle issues that make them very different of the classical ones.  To define the variable Lebesgue spaces $L^{p(\cdot)}(\Omega)$ with a measurable set $\Omega\subset\mathbb{R}^{n}$, we resort to a more subtle approach which is similar to that used to define the Luxemburg norm on the Orlicz spaces.
\smallbreak

Let $\mathcal{P}(\Omega)$ be the set of all Lebesgue measurable functions $p(\cdot):\Omega\rightarrow[1,+\infty]$.  We recall the following definition.

\begin{definition}\label{de1.1}
Given a  measurable set $\Omega\subset\mathbb{R}^{n}$ and $p(\cdot)\in \mathcal{P}(\Omega)$, for a measurable function $f$, we define
\begin{equation}\label{eq1.3}
   \|f\|_{L^{p(\cdot)}}:=\inf\left\{\lambda>0:  \rho_{p(\cdot)}\left(\frac{f}{\lambda}\right)\leq 1\right\},
\end{equation}
where the modular function $\rho_{p(\cdot)}$ associated with $p(\cdot)$ is given by
\begin{equation*}
   \rho_{p(\cdot)}(f):=\int_{\Omega}|f(x)|^{p(x)}dx.
\end{equation*}
Moreover, if the set on the right-hand side of  \eqref{eq1.3} is empty then we define $\|f\|_{L^{p(\cdot)}}=\infty$.
\end{definition}

For the classical Lebesgue space $L^{p}(\Omega)$ ($1\leq p<\infty$), its norm is directly defined  by
\begin{equation}\label{eq1.4}
\|f\|_{L^{p}}:=\left(\int_{\Omega}|f(x)|^{p}dx\right)^{\frac{1}{p}}.
\end{equation}
Note that if the exponent function $p(\cdot)$ is a constant, i.e. if $p(\cdot)=p\in [1,\infty)$, then we obtain the usual norm \eqref{eq1.3} via the modular function $ \rho_{p}$.
\smallbreak

\begin{definition}\label{de1.2}
Given a  measurable set $\Omega\subset\mathbb{R}^{n}$ and $p(\cdot)\in \mathcal{P}(\Omega)$, we define the variable exponent Lebesgue spaces $L^{p(\cdot)}(\Omega)$ to be the set of Lebesgue measurable functions $f$ such that   $\|f\|_{L^{p(\cdot)}}<+\infty$.
\end{definition}

 It is easy to verify that $L^{p(\cdot)}(\Omega)$ is a vector space, and the function $\|\cdot\|_{L^{p(\cdot)}}$ defines a norm on $L^{p(\cdot)}(\Omega)$, thus $L^{p(\cdot)}(\Omega)$ is a normed vector space. Actually, $L^{p(\cdot)}(\Omega)$ is a Banach space associated with the norm $\|\cdot\|_{L^{p(\cdot)}}$. For a complete  presentation of the theory of a  variable Lebesgue spaces, we refer to the readers to see books \cite{CF13, DHHR11}.
\smallbreak

Now we state the main results of this paper. The first result is the global existence of solutions of system \eqref{eq1.1} with small initial data in mixed variable Lebesgue space (see Section 2 for the definition of this space).

\begin{theorem}\label{th1.3}
Let $n\ge 2$, $1<\alpha\leq2$ such that $\frac{n}{2(\alpha-1)}>1$,  and let $p(\cdot)\in \mathcal{P}^{\log}(\mathbb{R}^{n})$ with $1<p^{-}\leq p(\cdot) \leq p^{+}<+\infty$. For any $\nabla(-\Delta)^{-1}u_{0}\in \mathcal{L}^{p(\cdot)}_{\frac{n}{\alpha-1}}(\mathbb{R}^{n})$, there exists a positive constant $\varepsilon$ such that if the initial data $u_{0}$ satisfies
\begin{equation}\label{eq1.5}
\|\nabla(-\Delta)^{-1}u_{0}\|_{\mathcal{L}^{p(\cdot)}_{\frac{n}{\alpha-1}}}\leq \varepsilon,
\end{equation}
then the system \eqref{eq1.1} admits a unique global solution $u$ such that
\begin{equation}\label{eq1.6}
\nabla(-\Delta)^{-1}u\in \mathcal{L}^{p(\cdot)}_{\frac{n}{\alpha-1}}\left(\mathbb{R}^{n},L^{\infty}(0,+\infty)\right).
\end{equation}
\end{theorem}

The main reason we considered the mixed variable Lebesgue space $\mathcal{L}^{p(\cdot)}_{\frac{n}{\alpha-1}}(\mathbb{R}^{n})$ is essentially technical and it is motivated by the lack of flexibility of the indices that intervene in the boundedness of the Riesz transforms.
Moreover, we should emphasize the fact that, in Theorem \ref{th1.3}, we first measure the behavior of solutions for the time variable $t$ in the usual $L^{\infty}$ space, then we measure the behavior of solution for the space variable $x$ in the mixed variable Lebesgue space $\mathcal{L}^{p(\cdot)}_{\frac{n}{\alpha-1}}(\mathbb{R}^{n})$.
\smallbreak

In our second result, we change the order of the variables: we first measure the behavior of solutions for the space variable $x$ in the usual Lebesgue space $L^{p}(\mathbb{R}^{n})$, then we measure the behavior of solution for the time variable $t$ in the variable Lebesgue space $L^{q(\cdot)}(0,T)$.

\begin{theorem}\label{th1.4}
Let $n\ge 2$, $1<\alpha\leq2$, $q(\cdot)\in \mathcal{P}^{\log}(0,+\infty)$ with $2<q^{-}\leq q(\cdot)\leq q^{+}<+\infty$, and $p>\frac{n}{\alpha}$ satisfying the relationship $\frac{\alpha}{q(\cdot)}+\frac{n}{p}<1$, and $\bar{p}(\cdot)\in \mathcal{P}^{\text{emb}}_{p}(\mathbb{R}^{n})$. Then for any $\nabla(-\Delta)^{-1}u_{0}\in L^{\bar{p}(\cdot)}(\mathbb{R}^{n})$, there exists a time $T>0$ such that the system \eqref{eq1.1} admits a unique local solution $u$ satisfying
\begin{equation}\label{eq1.7}
\nabla(-\Delta)^{-1}u\in L^{q(\cdot)}\left(0,T; L^{p}(\mathbb{R}^{n})\right).
\end{equation}
\end{theorem}

The rest of this paper is organized as follows. In Section 2, we present the detailed review of the properties of the variable exponent Lebesgue spaces and list all kinds of analytic estimates in terms of the bounded properties of singular integral operators in the variable Lebesgue spaces, moreover, we recall some decay estimates of the fractional heat kernels involved in the proof of our main results. In Sections 3 and 4,  we establish the linear and nonlinear estimates of solutions in each solution spaces, then complete the proofs of Theorem \ref{th1.3} and  Theorem \ref{th1.4} by the Banach fixed point theorem.

\section{ Preliminaries}

In this section we introduce some conventions and notations, and state some basic results. For the latter we refer to the books \cite{CF13,DHHR11}.

We use the notations $C$ as a generic harmless constant, i.e. a constant whose value may change from appearance to appearance. By $\chi_{\Omega}$ we denote the characteristic function of $\Omega\subset\mathbb{R}^n$. The notation $\mathcal{X}\hookrightarrow \mathcal{Y}$ denotes continuous embedding from $\mathcal{X}$ to $\mathcal{Y}$.  For any $p(\cdot)\in \mathcal{P}(\Omega)$, we denote
$$
 p^{-}:=\operatorname{essinf}_{x\in\Omega}p(x), \ \  p^{+}:=\operatorname{esssup}_{x\in\Omega}p(x).
$$
Throughout this paper, we will always assume that $1<p^{-}\leq p^{+}<+\infty$.
\smallbreak

\subsection{The variable Lebesgue spaces}

We collect some properties of the variable exponent Lebesgue spaces.  The first one is a generalization of the H\"{o}lder's inequality (cf. \cite{CF13}, Corollary 2.28; \cite{DHHR11}, Lemma 3.2.20).

\begin{lemma}\label{le2.1}
Given two exponent functions $p_{1}(\cdot), p_{2}(\cdot)\in \mathcal{P}(\Omega)$, define $p(\cdot)\in \mathcal{P}(\Omega)$ by $\frac{1}{p(x)}=\frac{1}{p_{1}(x)}+\frac{1}{p_{2}(x)}$. Then there exists a constant $C$ such that for all $f\in L^{p_{1}(\cdot)}(\Omega)$ and $g\in L^{p_{2}(\cdot)}(\Omega)$, we have $fg\in L^{p(\cdot)}(\Omega)$ and
\begin{equation}\label{eq2.1}
   \|fg\|_{L^{p(\cdot)}}\leq C \|f\|_{L^{p_{1}(\cdot)}}\|g\|_{L^{p_{2}(\cdot)}}.
\end{equation}
\end{lemma}

Note that in the classical Lebesgue $L^{p}(\Omega)$ ($1\leq p<+\infty$), the norm can be represented by using the norm conjugate formula
 \begin{equation}\label{eq2.2}
   \|f\|_{L^{p}}\leq \sup_{\|g\|_{L^{p'}}\leq1}\int_{\Omega}\left|f(x)g(x)\right|dx,
\end{equation}
where $1<p'\leq+\infty$ is a conjugate of $p$, i.e. $\frac{1}{p}+\frac{1}{p'}=1$.   A slightly weaker analog of the equality \eqref{eq2.2} is true for the variable Lebesgue spaces. Indeed, given $p(\cdot)\in \mathcal{P}(\Omega)$, we  define the conjugate exponent function of $p(\cdot)$ by the formula
$$
  \frac{1}{p(x)}+\frac{1}{p'(x)}=1, \ \ x\in\Omega.
$$
Then the norm $\|\cdot\|_{L^{p(\cdot)}}$ satisfies the following norm conjugate formula (cf. \cite{DHHR11}, Corollary 3.2.14).
\begin{lemma}\label{le2.2}
Let $p(\cdot)\in \mathcal{P}(\Omega)$, and let $p'(\cdot)$ be the conjugate of $p(\cdot)$. Then we have
\begin{equation}\label{eq2.3}
   \frac{1}{2}\|f\|_{L^{p(\cdot)}}\leq \sup_{\|g\|_{L^{p'(\cdot)}}\leq1}\int_{\Omega}\left|f(x)g(x)\right|dx\leq 2\|f\|_{L^{p(\cdot)}}.
\end{equation}
\end{lemma}

We know that every function in the  variable Lebesgue space is locally integrable, and it holds the following embedding result with a sharper embedding constant (cf. \cite{CF13}, Corollary 2.48).
\begin{lemma}\label{le2.3}
Given  two exponent functions $p_{1}(\cdot), p_{2}(\cdot)\in \mathcal{P}(\Omega)$ with $1<p_{1}^{+}, p_{2}^{+}<+\infty$. Then  $L^{p_{2}(\cdot)}(\Omega)\hookrightarrow L^{p_{1}(\cdot)}(\Omega)$ if and only if $p_{1}(x)\leq p_{2}(x)$ almost everywhere. Furthermore,  in this case we have
\begin{equation}\label{eq2.4}
   \|f\|_{L^{p_{1}(\cdot)}}\leq \left(1+|\Omega|\right)\|f\|_{L^{p_{2}(\cdot)}}.
\end{equation}
\end{lemma}

An interesting fact in the setting of variable Lebesgue spaces is given by the extension of the result presented in Lemma \ref{le2.3} to whole space $\mathbb{R}^n$.
\smallbreak
{Before presenting a result about it, 
we  must stress the fact that not all properties of the usual Lebesgue spaces $L^{p}(\Omega)$ can be generalized to the variable Lebesgue spaces. For example, the variable Lebesgue spaces are not translation invariant, thus the convolution of two functions $f$ and $g$ is not well-adapted, and the Young's inequality are not valid anymore (cf. \cite{CF13}, Section 5.3). In consequence, new ideas and techniques are need to tackle with the boundedness of many classical operators appeared in the mathematical analysis of PDEs. A classical approach to study these difficulties is to consider some constraints on the variable exponent, and the most common one is given by the so-called \textit{log-H\"{o}lder continuity condition}. We introduce the following definition.
\begin{definition}\label{de2.6}
Let $p(\cdot)\in \mathcal{P}(\mathbb{R}^{n})$ such that  there exists a limit $\frac{1}{p_{\infty}}=\lim_{|x|\rightarrow\infty}\frac{1}{p(x)}$.
\begin{itemize}
\item We say that $p(\cdot)$ is locally log-H\"{o}lder continuous if for all $x,y\in \mathbb{R}^{n}$,  there exists a constant $C$ such that $\big{|}\frac{1}{p(x)}-\frac{1}{p(y)}\big{|}\leq \frac{C}{\log(e+\frac{1}{|x-y|})}$;
\item We say that $p(\cdot)$ satisfies the log-H\"{o}lder decay condition if for all $x\in \mathbb{R}^{n}$,  there exists a constant $C$ such that $\big{|}\frac{1}{p(x)}-\frac{1}{p_{\infty}}\big{|}\leq \frac{C}{\log(e+|x|)}$;
\item We say that  $p(\cdot)$ is globally log-H\"{o}lder continuous in $\mathbb{R}^n$ if it is locally \textit{log-H\"{o}lder continuous} and satisfies the \textit{log-H\"{o}lder decay condition};
\item We define the class of variable exponents $\mathcal{P}^{\log}(\mathbb{R}^{n})$ as
\begin{equation*}
 \mathcal{P}^{\log}(\mathbb{R}^{n}):=\left\{p(\cdot)\in\mathcal{P}(\mathbb{R}^{n}):\ \  p(\cdot)\ \text{is globally log-H\"{o}lder continuous in}\  \mathbb{R}^n\right\}.
\end{equation*}
\end{itemize}
\end{definition}
With this notion of variable exponent at hand, we characterize a class of variable exponents for which is possible to obtain an analogous result to Lemma \ref{le2.3} in the setting of unbounded domains.
\begin{definition}\label{de2.4}
Given a constant exponent
$p\in (1,+\infty)$, we define the class of variable exponents $\mathcal{P}^{\text{emb}}_p(\mathbb{R}^n)$  as the set
\begin{equation*}
  \mathcal{P}^{\text{emb}}_p(\mathbb{R}^n)\! :=\!\left\{\! \bar{p}(\cdot)\in\mathcal{P}^{\log}(\mathbb{R}^n):
  p\leq(\bar{p})^{-} \leq (\bar{p})^{+}\! <\!+\infty\ \text{ and } \
  \frac{ p\bar{p}(x)}{ \bar{p}(x)- p}\to +\infty \text{ as } |x|\to +\infty \!\right\}.
\end{equation*}
\end{definition}

Then, a consequence of considering a variable exponent  in $\mathcal{P}^{\text{emb}}_p(\mathbb{R}^n)$
is given by following result   (cf. Theorem 2.45 and Remark 2.46 in \cite{CF13}).

\begin{lemma}\label{le2.5}
    Let $p\in (1,+\infty)$ and  $\bar{p}(\cdot)\in\mathcal{P}^{\text{emb}}_p(\mathbb{R}^n)$. Then we have
    $L^{\bar{p}(\cdot)}(\mathbb{R}^n)\hookrightarrow L^{p}(\mathbb{R}^n)$, and there exists a constant $C>0$ such that
\begin{equation}\label{eq2.5}
    \|f\|_{L^{p}} \leq C\|f\|_{L^{\bar{p}(\cdot)}}.
\end{equation}
\end{lemma}
}

On the other hand, for any $p(\cdot)\in \mathcal{P}^{\log}(\mathbb{R}^{n})$, we have the following results in terms of the Hardy-Littlewood maximal function and  Riesz transforms (cf. \cite{CF13}, Theorem 3.16 and Theorem 5.42;\cite{DHHR11}, Theorem 4.3.8 and Corollary 6.3.10).

\begin{lemma}\label{le2.7}
Let $p(\cdot)\in \mathcal{P}^{\log}(\mathbb{R}^{n})$ with $1<p^{-}\leq p^{+}<+\infty$. Then  for any $f\in L^{p(\cdot)}(\mathbb{R}^{n})$, there exists a positive constant $C$ such that
\begin{equation}\label{eq2.6}
   \|\mathcal{M} (f)\|_{L^{p(\cdot)}}\leq C\|f\|_{L^{p(\cdot)}},
\end{equation}
where $\mathcal{M}$ is the Hardy-Littlewood maximal function defined by
\begin{equation*}
  \mathcal{M}(f)(x):=\sup_{x\in B}\frac{1}{|B|}\int_{B}|f(y)|dy,
\end{equation*}
and $B\subset \mathbb{R}^{n}$ is an open ball with center $x$. Furthermore,
\begin{equation}\label{eq2.7}
   \|\mathcal{R}_{j}(f)\|_{L^{p(\cdot)}}\leq C\|f\|_{L^{p(\cdot)}} \ \ \text{for any}\ \ 1\leq j\leq n,
\end{equation}
where $\mathcal{R}_{j}$ ($1\leq j\leq n$) are the usual Riesz transforms, i.e. $\mathcal{F}\left(\mathcal{R}_{j}f\right)(\xi)=-\frac{i\xi_{j}}{|\xi|}\mathcal{F}(f)(\xi)$.
\end{lemma}

We also need to use the boundedness of the Riesz potential in the variable Lebesgue spaces.

\begin{definition}\label{de2.8}
Given $0<\beta <n$, for any measurable function $f$, define the Riesz potential $\mathcal{I}_{\beta}$ also referred to as the fractional integral operator with index $\beta$, to be the convolution operator
\begin{equation}\label{eq2.8}
   \mathcal{I}_{\beta}(f)(x):=C(\beta, n)\int_{\mathbb{R}^{n}}\frac{|f(y)|}{|x-y|^{n-\beta}}dy,
\end{equation}
where
\begin{equation*}
   C(\beta, n)=\frac{\Gamma(\frac{n-\beta}{2})}{\pi^{\frac{n}{2}}2^{\beta}\Gamma(\frac{\beta}{2})}.
\end{equation*}
\end{definition}

The Riesz potential is well defined on the variable Lebesgue spaces, and if $p_{+}<\frac{n}{\beta}$ and $f\in L^{p(\cdot)}(\mathbb{R}^{n})$, then $\mathcal{I}_{\beta}(f)(x)$ converges for every $x$. Moreover, we have the following boundedness result (cf. \cite{CF13}, Theorem 5.46).

\begin{lemma}\label{le2.9}
Let $p(\cdot)\in \mathcal{P}^{\log}(\mathbb{R}^{n})$ with $1<p^{-}\leq p^{+}<+\infty$, and let $0<\beta<\frac{n}{p^{+}}$. Then for any $f\in L^{p(\cdot)}(\mathbb{R}^{n})$, there exists a positive constant $C$ such that
\begin{equation}\label{eq2.9}
   \|\mathcal{I}_{\beta}(f)\|_{L^{q(\cdot)}}\leq C\|f\|_{L^{p(\cdot)}} \ \ \text{with}\ \ \frac{1}{q(\cdot)}=\frac{1}{p(\cdot)}-\frac{\beta}{n}.
\end{equation}
\end{lemma}

According to the estimate \eqref{eq2.9}, there is a very strong relationship between the variable indices $p(\cdot)$ and $q(\cdot)$, which leads to a strict restriction to apply Lemma \ref{le2.9}. In order to relax the applicable scope of \eqref{eq2.9}, we introduce the definition of the mixed Lebesgue spaces (cf. \cite{C22}).

\begin{definition}\label{de2.10}
Let $p(\cdot)\in \mathcal{P}^{\log}(\mathbb{R}^{n})$ with $1<p^{-}\leq p^{+}<+\infty$, and let $1<r<+\infty$.  Then we define the mixed Lebesgue space $\mathcal{L}^{p(\cdot)}_{r}(\mathbb{R}^{n})$ as the intersection of the spaces $L^{p(\cdot)}(\mathbb{R}^{n})$ and $L^{r}(\mathbb{R}^{n})$,  i.e.
\begin{equation*}
   \mathcal{L}^{p(\cdot)}_{r}(\mathbb{R}^{n}):= L^{p(\cdot)}(\mathbb{R}^{n})\bigcap L^{r}(\mathbb{R}^{n}),
\end{equation*}
which is endowed with the norm
\begin{equation*}
  \|f\|_{\mathcal{L}^{p(\cdot)}_{r}}:= \max \{\|f\|_{L^{p(\cdot)}},\|f\|_{L^{r}}\}.
\end{equation*}
\end{definition}

The mixed Lebesgue space $\mathcal{L}^{p(\cdot)}_{r}(\mathbb{R}^{n})$ clearly inherits some properties of the spaces   $L^{p(\cdot)}(\mathbb{R}^{n})$ and $L^{r}(\mathbb{R}^{n})$. For example, we have the H\"{o}lder's inequality
\begin{equation}\label{eq2.10}
   \|fg\|_{\mathcal{L}^{p(\cdot)}_{r}}\leq C \|f\|_{\mathcal{L}^{p_{1}(\cdot)}_{r_{1}}}\|g\|_{\mathcal{L}^{p_{2}(\cdot)}_{r_{2}}}
\end{equation}
with $\frac{1}{p(\cdot)}=\frac{1}{p_{1}(\cdot)}+\frac{1}{p_{2}(\cdot)}$ and $\frac{1}{r}=\frac{1}{r_{1}}+\frac{1}{r_{2}}$, and the Riesz transforms are also bounded in these spaces, i.e.
\begin{equation}\label{eq2.11}
   \|\mathcal{R}_{j}(f)\|_{\mathcal{L}^{p(\cdot)}_{r}}\leq C\|f\|_{\mathcal{L}^{p(\cdot)}_{r}} \ \ \text{for any}\ \ 1\leq j\leq n.
\end{equation}
Furthermore, we have the following  mixed Hardy-Littlewood-Sobolev inequality (cf. \cite{C22}, Theorem 4) .

\begin{lemma}\label{le2.11}
Let $p(\cdot)\in \mathcal{P}^{\log}(\mathbb{R}^{n})$ with $1<p^{-}\leq p^{+}<+\infty$, and let $1<r<+\infty$. Then for any $0<\beta<\min\{\frac{n}{p^{+}}, \frac{n}{r}\}$, and $f\in \mathcal{L}^{p(\cdot)}_{r}(\mathbb{R}^{n})$, there exists a positive constant $C$ such that
\begin{equation}\label{eq2.12}
   \|\mathcal{I}_{\beta}(f)\|_{L^{q(\cdot)}}\leq C\|f\|_{\mathcal{L}^{p(\cdot)}_{r}} \ \ \text{with}\ \ q(\cdot)=\frac{np(\cdot)}{n-\beta r}.
\end{equation}
\end{lemma}

 Notice that if we set $p(\cdot)=r$ we recover the inequality \eqref{eq2.9} in the framework of the usual Lebesgue spaces. Moreover, since the number $r$ does not depend on $p^{-}$, $p^{+}$  and $p(x)$,  the condition in \eqref{eq2.12} gives us more flexibility than the condition in \eqref{eq2.9}.
\smallbreak

\subsection{Estimates of the Fractional Heat Kernel}

Let us recall some crucial estimates of the fractional heat kernel $G^{\alpha}_{t}(x)$ involved in the integral formulation of system \eqref{eq1.1}, where
\begin{equation*}
G^{\alpha}_{t}(x)=\mathcal{F}^{-1}(e^{-t|\xi|^{\alpha}})
=(2\pi)^{-\frac{n}{2}}\int_{\mathbb{R}^{n}}e^{ix\cdot\xi}e^{-t|\xi|^{\alpha}}d\xi.
\end{equation*}
The first one is the point-wise estimate of the fractional heat kernel $G^{\alpha}_{t}(x)$ (cf. \cite{MYZ08}, Remark 2.1).

\begin{lemma}\label{le2.12}
For $\alpha>0$,  the kernel function $G^{\alpha}(x)$ satisfies the following point-wise estimate
\begin{align}\label{eq2.13}
|\nabla G^{\alpha}_{t}(x)|\leq C(t^{\frac{1}{\alpha}}+|x|)^{-n-1}, \ \ x\in\mathbb{R}^{n}.
\end{align}
\end{lemma}

The second one is a classical lemma in terms of the Hardy--Littlewood maximal function (cf. \cite{L18}, Lemma 7.4).

\begin{lemma}\label{le2.13}
Let $\varphi$ be a radially decreasing function on $\mathbb{R}^{n}$ and $f$ a locally integrable function. Then we have
\begin{equation}\label{eq2.14}
   |(\varphi\ast f)(x)|\leq\|\varphi\|_{L^{1}}\mathcal{M}(f)(x).
\end{equation}
\end{lemma}

Finally, we recall the classical $L^{p}$- $L^{q}$ estimates of the fractional heat kernel $G^{\alpha}_{t}(x)$ (cf. \cite{MYZ08}, Lemma 3.1).
\begin{lemma}\label{le2.14}
For all $\alpha>0$, $\nu>0$, $1\leq p\leq q\leq\infty$. Then for any $f\in L^{p}(\mathbb{R}^{n})$, we have
\begin{align}\label{eq2.15}
   &\|G^{\alpha}_{t}\ast f\|_{L^{q}}\leq Ct^{-\frac{n}{\alpha}(\frac{1}{p}-\frac{1}{q})}\|f\|_{L^{p}},\\
   \label{eq2.16}
      &\|\Lambda^{\nu}G^{\alpha}_{t}\ast f\|_{L^{q}}\leq Ct^{-\frac{\nu}{\alpha}-\frac{n}{\alpha}(\frac{1}{p}-\frac{1}{q})}\|f\|_{L^{p}}.
\end{align}

\end{lemma}

\section{Proof of Theorem \ref{th1.3}}

In this section, we shall show Theorem \ref{th1.3} by applying the following existence and uniqueness result for an abstract operator equation in a generic Banach space (cf. \cite{L02}, Theorem 13.2).
\begin{proposition}\label{pro3.1}
Let $\mathcal{X}$ be a Banach space and
$\mathcal{B}:\mathcal{X}\times\mathcal{X}\rightarrow\mathcal{X}$  a
bilinear bounded operator. Assume that for any $u,v\in
\mathcal{X}$, we have
$$
  \|\mathcal{B}(u,v)\|_{\mathcal{X}}\leq
  C\|u\|_{\mathcal{X}}\|v\|_{\mathcal{X}}.
$$
Then for any $y\in \mathcal{X}$ such that $\|y\|_{\mathcal{X}}\leq
\eta<\frac{1}{4C}$, the equation $u=y+\mathcal{B}(u,u)$ has a
solution $u$ in $\mathcal{X}$. Moreover, this solution is the only
one such that $\|u\|_{\mathcal{X}}\leq 2\eta$, and depends
continuously on $y$ in the following sense: if
$\|\widetilde{y}\|_{\mathcal{X}}\leq \eta$,
$\widetilde{u}=\widetilde{y}+\mathcal{B}(\widetilde{u},\widetilde{u})$
and $\|\widetilde{u}\|_{\mathcal{X}}\leq 2\eta$, then
$$
  \|u-\widetilde{u}\|_{\mathcal{X}}\leq \frac{1}{1-4\eta
  C}\|y-\widetilde{y}\|_{\mathcal{X}}.
$$
\end{proposition}

In order to do so, we need to transform  the system \eqref{eq1.1} into an equivalent system. Firstly, notice that the chemical concentration $\phi$ is determined by the Poisson equation, the second equation of \eqref{eq1.1}, gives rise to the coefficient $\nabla\phi$ in the first equation of \eqref{eq1.1},  thus $\phi$ can be represented as the volume potential of $u$:
\begin{equation*}
\phi(t,x)=(-\Delta)^{-1}u(t,x)=
\begin{cases}
\displaystyle
\frac{1}{n(n-2)\omega_{n}}\int_{\mathbb{R}^{n}}\frac{u(t,y)}{|x-y|^{n-2}}dy,
&\quad \text{if }n\geq 3,\\[10pt]
\displaystyle-\frac{1}{2\pi}\int_{\mathbb{R}^{2}}u(t,y)\log|x-y|dy,
&\quad \text{if }n=2,
\end{cases}
\end{equation*}
where $\omega_{n}$ denotes the surface area of the unit sphere in
$\mathbb{R}^{n}$. Secondly, observe carefully that the nonlinear term $u\nabla(-\Delta)^{-1}u$ has a nice symmetric structure:
\begin{equation}\label{eq3.1}
    u\nabla(-\Delta)^{-1}u=-\nabla\cdot\left(\nabla(-\Delta)^{-1}u\otimes\nabla(-\Delta)^{-1}u-\frac{1}{2}\left|\nabla(-\Delta)^{-1}u\right|^{2}I\right),
\end{equation}
where $\otimes$ is a tensor product, and $I$ is a $n$-th order identity matrix, thus if we denote $v:=\nabla(-\Delta)^{-1}u$, and  taking the operator $\nabla(-\Delta)^{-1}$ to both sides of the system \eqref{eq1.1}, then we can further reduced the system \eqref{eq1.1} into the following system:
\begin{equation}\label{eq3.2}
\begin{cases}
  \partial_{t}v-\Delta v-\mathcal{R}\otimes\mathcal{R}\nabla\cdot\left(v\otimes v-\frac{1}{2}\left|v\right|^{2}I\right)=0,\\
    v(0, x)=v_0(x),
\end{cases}
\end{equation}
where $v_0:=\nabla(-\Delta)^{-1}u_{0}$, $\mathcal{R}:=(\mathcal{R}_{1}, \mathcal{R}_{2}, \cdots, \mathcal{R}_{n})$ and $\mathcal{R}_{j}$ ($j=1,2,\cdots,n$) are Riesz transforms.
Finally, based on the framework of the Kato's analytical semigroup, we can rewrite system \eqref{eq3.2} as an equivalent integral form:
\begin{equation}\label{eq3.3}
  v(t)=G^{\alpha}_{t}\ast v_{0}(x)+\int_{0}^{t}G^{\alpha}_{t-s}\ast\mathcal{R}\otimes\mathcal{R}\nabla\cdot\left(v\otimes v-\frac{1}{2}\left|v\right|^{2}I\right)ds.
\end{equation}

Now, let the assumptions of Theorem \ref{th1.3} be in force, we introduce the space $\mathcal{X}$ as follows:
$$
\mathcal{X}:=\mathcal{L}^{p(\cdot)}_{\frac{n}{\alpha-1}}(\mathbb{R}^{n}; L^{\infty}(0,\infty)),
$$
and the norm of the  space $\mathcal{X}$ is endowed by
\begin{equation*}
   \|v\|_{\mathcal{X}}=\max\big{\{}\|v\|_{L^{p(\cdot)}_{x}(L^{\infty}_{t})}, \|v\|_{L^{\frac{n}{\alpha-1}}_{x}(L^{\infty}_{t})}\big{\}}.
\end{equation*}
It can be easily seen that $(\mathcal{X},  \|\cdot\|_{\mathcal{X}})$ is a Banach space.
\smallbreak

In the sequel, we shall establish the linear and nonlinear estimates of the integral equation \eqref{eq3.3} in the space $\mathcal{X}$, respectively.
\begin{lemma}\label{le3.2}
Let the assumptions of Theorem \ref{th1.3} be in force. Then for any $v_{0}\in \mathcal{L}^{p(\cdot)}_{\frac{n}{\alpha-1}}(\mathbb{R}^{n})$,
there exists a positive constant $C$ such that
\begin{equation}\label{eq3.4}
   \|G^{\alpha}_{t}\ast v_{0}\|_{\mathcal{X}_{T}}\leq
   C\|v_{0}\|_{\mathcal{L}^{p(\cdot)}_{\frac{n}{\alpha-1}}}.
\end{equation}
\end{lemma}
\begin{proof}
Since the fractional heat kernel $G^{\alpha}_{t}(x)$ is a radially decreasing function with respect to the space variable, we can use Lemma \ref{le2.13} to get
\begin{equation*}
   \left|G^{\alpha}_{t}\ast v_{0}\right|\leq \|G_{t}(\cdot)\|_{L^{1}_{x}}\mathcal{M}(v_{0})(x)=\mathcal{M}(v_{0})(x),
\end{equation*}
which directly yields that
\begin{equation}\label{eq3.5}
   \|G^{\alpha}_{t}\ast v_{0}\|_{L^{\infty}_{t}}\leq \mathcal{M}(v_{0})(x).
\end{equation}
Then, by taking the $\mathcal{L}^{p(\cdot)}_{\frac{n}{\alpha-1}}$-norm to both sides of \eqref{eq3.5} with respect to the space variable $x$,  we see that
\begin{equation}\label{eq3.6}
   \|G^{\alpha}_{t}\ast v_{0}\|_{\mathcal{X}_{T}}\leq \|\mathcal{M}(v_{0})\|_{\mathcal{L}^{p(\cdot)}_{\frac{n}{\alpha-1}}}\leq \max\{\|\mathcal{M}(v_{0})\|_{L^{p(\cdot)}_{x}}, \|\mathcal{M}(v_{0})\|_{L^{\frac{n}{\alpha-1}}_{x}}\}.
\end{equation}
Recalling that $p(\cdot)\in\mathcal{P}^{\log}(\mathbb{R}^{n})$ with  $1<p^{-}\leq p^{+}<+\infty$, thus we know from Lemma \ref{le2.7} that the maximal function $\mathcal{M}$ is bounded in the variable Lebesgue space $L^{p(\cdot)}(\mathbb{R}^{n})$, as well as the usual Lebesgue space $L^{\frac{n}{\alpha-1}}(\mathbb{R}^{n})$, thus we obtain
\begin{align*}
   \|G^{\alpha}_{t}\ast v_{0}\|_{\mathcal{X}_{T}}&\leq \max\{\|\mathcal{M}(v_{0})\|_{L^{p(\cdot)}_{x}}, \|\mathcal{M}(v_{0})\|_{L^{\frac{n}{\alpha-1}}_{x}}\}\nonumber\\
      &\leq C\max\{\|v_{0}\|_{L^{p(\cdot)}_{x}}, \|v_{0}\|_{L^{\frac{n}{\alpha-1}}_{x}}\}\nonumber\\
      &=C\|v_{0}\|_{\mathcal{L}^{p(\cdot)}_{\frac{n}{\alpha-1}}}.
\end{align*}
Thus, we complete the proof of Lemma \ref{le3.2}. 
\end{proof}
\smallbreak

\begin{lemma}\label{le3.3}
Let the assumptions of Theorem \ref{th1.3} be in force.  Then
there exists a positive constant $C$ such that
\begin{equation}\label{eq3.7}
   \left\|\int_{0}^{t}G^{\alpha}_{t-s}\ast\mathcal{R}\otimes\mathcal{R}\nabla\cdot\left(v\otimes v-\frac{1}{2}\left|v\right|^{2}I\right)ds\right\|_{\mathcal{X}}\leq
   C \|v\|_{\mathcal{X}}^{2}.
\end{equation}
\end{lemma}
\begin{proof}
We begin by noticing that, due to the properties of Riesz transforms $\mathcal{R}$, for any $0<t<+\infty$, we can rewrite the second term on the right-hand side of
\eqref{eq3.3} as
\begin{align}\label{eq3.8}
   \int_{0}^{t}G^{\alpha}_{t-s}\ast\mathcal{R}\otimes\mathcal{R}\nabla\cdot\left(v\otimes v-\frac{1}{2}\left|v\right|^{2}I\right)ds=\mathcal{R}\otimes\mathcal{R}\int_{0}^{t}G^{\alpha}_{t-s}\ast\nabla\cdot\left(v\otimes v-\frac{1}{2}\left|v\right|^{2}I\right)ds.
\end{align}
By the Minkowski's integral inequality and the decay properties of the fractional heat kernel, we obtain
\begin{align*}
  \left|\int_{0}^{t}G^{\alpha}_{t-s}\ast\nabla\cdot\left(v\otimes v-\frac{1}{2}\left|v\right|^{2}I\right)ds\right|&\leq \int_{0}^{t}\int_{\mathbb{R}^{n}}\left|\nabla G^{\alpha}_{t-s}(x-y)\right| \left|v(s,x)\right|^{2}dyds\nonumber\\
   &\leq \int_{\mathbb{R}^{n}}\int_{0}^{t} |\nabla G^{\alpha}_{t-s}(x-y)|\left|v(s,x)\right|^{2}dsdy,
\end{align*}
which considering the $L^{\infty}$-norm with respect to the time variable and using Lemma \ref{le2.12}, one gets
\begin{align}\label{eq3.9}
    \left|\int_{0}^{t}G^{\alpha}_{t-s}\ast\nabla\cdot\left(v\otimes v-\frac{1}{2}\left|v\right|^{2}I\right)ds\right|\leq\int_{\mathbb{R}^{n}}\int_{0}^{t} \frac{1}{(|t-s|^{\frac{1}{\alpha}}+|x-y|)^{n+1}}ds\|v(\cdot,y)\|_{L^{\infty}_{t}}^{2}dy.
\end{align}
It is easy to calculate that
\begin{align}\label{eq3.10}
   \int_{0}^{t} \frac{1}{(|t-s|^{\frac{1}{\alpha}}+|x-y|)^{n+1}}ds&\leq \int_{0}^{+\infty} \frac{1}{(s^{\frac{1}{\alpha}}+|x-y|)^{n+1}}ds\nonumber\\
   &=\int_{0}^{+\infty} \frac{|x-y|^{\alpha}}{((|x-y|^{\alpha}\tau)^{\frac{1}{\alpha}}+|x-y|)^{n+1}}d\tau\nonumber\\
   &=\frac{1}{|x-y|^{n+1-\alpha}}\int_{0}^{+\infty} \frac{1}{(1+\tau^{\frac{1}{\alpha}})^{n+1}}d\tau\nonumber\\
   &\leq \frac{C}{|x-y|^{n+1-\alpha}}.
\end{align}
Thus taking \eqref{eq3.10} into \eqref{eq3.9} and using the definition of Riesz potential, we obtain
\begin{align}\label{eq3.11}
   \left|\int_{0}^{t}G^{\alpha}_{t-s}\ast\nabla\cdot\left(v\otimes v-\frac{1}{2}\left|v\right|^{2}I\right)ds\right|&\leq C\int_{\mathbb{R}^{n}}\frac{1}{|x-y|^{n+1-\alpha}}\|v(\cdot,y)\|_{L^{\infty}_{t}}^{2}dy\nonumber\\
   &=C\mathcal{I}_{\alpha-1}\left(\|v\|_{L^{\infty}_{t}}^{2}\right)(x).
\end{align}
Back to \eqref{eq3.8}, we know from
\eqref{eq3.11} that
\begin{align}\label{eq3.12}
 \mathcal{R}\otimes\mathcal{R}\int_{0}^{t}G^{\alpha}_{t-s}\ast\nabla\cdot\left(v\otimes v-\frac{1}{2}\left|v\right|^{2}I\right)ds\leq C\mathcal{R}\otimes\mathcal{R}\left(\mathcal{I}_{\alpha-1}\left(\|v\|_{L^{\infty}_{t}}^{2}\right)(x)\right).
\end{align}
By taking $\mathcal{L}^{p(\cdot)}_{\frac{n}{\alpha-1}}$-norm to both sides of \eqref{eq3.12} with respect to the space variable $x$, then using  Lemma \ref{le2.11}, the fact that Riesz transforms are bounded  in the spaces $\mathcal{L}^{p(\cdot)}_{\frac{n}{\alpha-1}}(\mathbb{R}^{n})$, \eqref{eq2.10} and \eqref{eq2.11}, we obtain
\begin{align}\label{eq3.13}
  \left\| \mathcal{R}\otimes\mathcal{R}\int_{0}^{t}G^{\alpha}_{t-s}\ast\nabla\cdot\left(v\otimes v-\frac{1}{2}\left|v\right|^{2}I\right)ds\right\|_{L^{p(\cdot)}_{x}(L^{\infty}_{t})}&\leq C\left\|\mathcal{I}_{\alpha-1}\left(\|v\|_{L^{\infty}_{t}}^{2}\right)\right\|_{L^{p(\cdot)}_{x}}\nonumber\\
  &\leq C\|\|v\|_{L^{\infty}_{t}}^{2}\|_{\mathcal{L}^{\frac{p(\cdot)}{2}}_{\frac{n}{2(\alpha-1)}}}\nonumber\\
  &\leq C\|v\|_{\mathcal{L}^{p(\cdot)}_{\frac{n}{\alpha-1}}(L^{\infty}_{t})}^{2};
\end{align}
\begin{align}\label{eq3.14}
  \left\| \mathcal{R}\otimes\mathcal{R}\int_{0}^{t}G^{\alpha}_{t-s}\ast\nabla\cdot\left(v\otimes v-\frac{1}{2}\left|v\right|^{2}I\right)ds\right\|_{L^{\frac{n}{\alpha-1}}_{x}(L^{\infty}_{t})}&\leq C\left\|\mathcal{I}_{\alpha-1}\left(\|v\|_{L^{\infty}_{t}}^{2}\right)\right\|_{L^{\frac{n}{\alpha-1}}_{x}}\nonumber\\
  &\leq C\|\|v\|_{L^{\infty}_{t}}^{2}\|_{L^{\frac{n}{2(\alpha-1)}}}\nonumber\\
  &\leq C\|v\|_{L^{\frac{n}{\alpha-1}}(L^{\infty}_{t})}^{2}.
\end{align}
Putting the above estimates \eqref{eq3.13} and \eqref{eq3.14} together, we get \eqref{eq3.7}. Thus, we complete the proof of Lemma \ref{le3.3}.
\end{proof}
\smallbreak

Based on the desired linear and nonlinear estimates obtained in Lemmas \ref{le3.2} and  \ref{le3.3}, we know that there exist two positive
constants $C_{1}$ and $C_{2}$ such that
\begin{align}\label{eq3.15}
  \|v_{0}\|_{\mathcal{X}}\leq
 C_{1}\|v_{0}\|_{\mathcal{L}^{p(\cdot)}_{\frac{n}{\alpha-1}}}+C_{2}\|v\|_{\mathcal{X}}^2.
\end{align}
Thus if the initial data $v_{0}$ satisfies the condition
$$
\|v_{0}\|_{\mathcal{L}^{p(\cdot)}_{\frac{n}{\alpha-1}}}\leq \frac{1}{4C_{1}C_{2}},
$$
we can apply Proposition \ref{pro3.1} to get global existence of solution $v\in \mathcal{X}$ to the system \eqref{eq3.2}. This yields a unique global solution $u$ of the system \eqref{eq1.1}  such that $\nabla(-\Delta)^{-1}u\in\mathcal{X}$. Thus, we complete the proof of Theorem \ref{th1.3}.

\section{Proof of Theorem \ref{th1.4}}

In this section we present the proof of Theorem \ref{th1.4}. In this case, under the assumptions of Theorem \ref{th1.4}, we choose the solution space as
\begin{equation*}
   \mathcal{Y}_{T}:=L^{q(\cdot)}(0,T; L^{p}(\mathbb{R}^n)),
\end{equation*}
where $T>0$ is a constant to be determined later.  We endow the norm of the  space $\mathcal{Y}_{T}$ with the Luxemburg norm as
\begin{equation*}
   \|f\|_{\mathcal{Y}_{T}}=\inf\left\{\lambda>0: \int_{0}^{T}\left|\frac{\|f(t,\cdot)\|_{L^p}}{\lambda}\right|^{q(t)}dt\leq1\right\}.
\end{equation*}
\smallbreak

Similarly, we need to establish the linear and nonlinear estimates of the integral equation \eqref{eq3.3} in the space $\mathcal{Y}_{T}$.
\begin{lemma}\label{le4.1}
Let the assumptions of Theorem \ref{th1.4} be in force.  Then for any $\theta_{0}\in L^{\bar{p}(\cdot)}(\mathbb{R}^{n})$,
there exists a positive constant $C$ such that
\begin{equation}\label{eq4.1}
   \|G^{\alpha}_{t}\ast v_{0}\|_{\mathcal{Y}_{T}}\leq
   C\max\{T^{\frac{1}{q^{-}}}, T^{\frac{1}{q^{+}}}\}\|v_{0}\|_{L^{\bar{p}(\cdot)}_{x}}.
\end{equation}
\end{lemma}
\begin{proof}
We infer from Lemma \ref{le2.14} that
\begin{equation}\label{eq4.2}
   \|G^{\alpha}_{t}\ast v_{0}\|_{L^{p}_{x}}\leq
   \|G^{\alpha}_{t}(x)\|_{L^{1}_{x}}\|v_{0}\|_{L^{p}_{x}}=\|v_{0}\|_{L^{p}_{x}}.
\end{equation}
To continue, we need to use a simple result in the context of variable Lebesgue spaces:
\begin{equation}\label{eq4.3}
   \|1\|_{L^{q(\cdot)}_{t}}\leq
   2\max\{T^{\frac{1}{q^{-}}}, T^{\frac{1}{q^{+}}}\}.
\end{equation}
Thus, by taking the $L^{q(\cdot)}$-norm to both sides of \eqref{eq4.2} with respect to the time variable $t$, and using the H\"{o}lder's inequality \eqref{eq2.1} yields that
\begin{equation}\label{eq4.4}
   \|G^{\alpha}_{t}\ast v_{0}\|_{L^{q(\cdot)}_{t}(L^{p}_{x})}\leq C\|1\|_{L^{q(\cdot)}_{t}}
   \|v_{0}\|_{L^{p}_{x}}\leq C\max\{T^{\frac{1}{q^{-}}}, T^{\frac{1}{q^{+}}}\}\|v_{0}\|_{L^{p}_{x}},
\end{equation}
By Lemma \ref{le2.5}, we finally conclude that
\begin{equation}\label{eq4.5}
   \|G^{\alpha}_{t}\ast v_{0}\|_{L^{q(\cdot)}_{t}(L^{p}_{x})}\leq  C\max\{T^{\frac{1}{q^{-}}}, T^{\frac{1}{q^{+}}}\}\|v_{0}\|_{L^{\bar{p}(\cdot)}_{x}}.
\end{equation}
Thus, we complete the proof of Lemma \ref{le4.1}.
\end{proof}
\smallbreak

\begin{lemma}\label{le4.2}
Let the assumptions of Theorem \ref{th1.4} be in force. Then for any $T>0$,
there exists a positive constant $C$ such that
\begin{equation}\label{eq4.6}
   \left\|\int_{0}^{t}G^{\alpha}_{t-s}\ast\mathcal{R}\otimes\mathcal{R}\nabla\cdot\left(v\otimes v-\frac{1}{2}\left|v\right|^{2}I\right)ds\right\|_{\mathcal{Y}_{T}}\leq
    C (1+T)\|\theta\|_{\mathcal{Y}_{T}}^{2}.
\end{equation}
\end{lemma}
\begin{proof}
For any $0<t\leq T$, taking $L^{p}$-norm to the second term on the right-hand side of \eqref{eq3.3} with respect to the space variable $x$,  and using Lemma \ref{le2.14} and the fact that the Riesz transforms are bounded in the space $L^{p}$, we obtain
\begin{align}\label{eq4.7}
     \left\|\int_{0}^{t}G^{\alpha}_{t-s}\ast\mathcal{R}\otimes\mathcal{R}\nabla\cdot\left(v\otimes v-\frac{1}{2}\left|v\right|^{2}I\right)ds\right\|_{L^{p}_{x}}&\leq \int_{0}^{t}\|\nabla G^{\alpha}_{t-s}\ast(v\otimes v+\frac{1}{2}|v|^{2}I)(s,\cdot)\|_{L^{p}_{x}}ds\nonumber\\
   &\leq \int_{0}^{t}\frac{1}{(t-s)^{\frac{1}{\alpha}+\frac{n}{\alpha p}}}\|v(s,\cdot)\|_{L^{p}_{x}}^{2}ds.
\end{align}
Then taking $L^{q(\cdot)}$-norm on the right-hand side of \eqref{eq4.7} with respect to the time variable $t$ and using the norm conjugate formula \eqref{eq2.3} by $\frac{1}{q(\cdot)}+\frac{1}{q'(\cdot)}=1$, we see that
\begin{align}\label{eq4.8}
    \left\|\int_{0}^{t}\frac{1}{(t-s)^{\frac{1}{\alpha}+\frac{2}{\alpha p}}}\|v(s,\cdot)\|_{L^{p}_{x}}^{2}ds\right\|_{L^{q(\cdot)}_{t}}
   &\leq 2\sup_{\|\psi\|_{L^{q'(\cdot)}_{t}}\leq 1}\int_{0}^{T}\int_{0}^{t}\frac{|\psi(t)|}{|t-s|^{\frac{1}{\alpha}+\frac{n}{\alpha p}}}\|v(s,\cdot)\|_{L^{p}_{x}}^{2}dsdt\nonumber\\
   &=2\sup_{\|\psi\|_{L^{q'(\cdot)}_{t}}\leq 1}\int_{0}^{T}\int_{0}^{T}\frac{\chi_{\{0<s<t\}}|\psi(t)|}{|t-s|^{\frac{1}{\alpha}+\frac{n}{\alpha p}}}dt\|v(s,\cdot)\|_{L^{p}_{x}}^{2}ds.
\end{align}
In order to use the 1D Riesz potential formula \eqref{eq2.8}, we extend the function $\psi(t)$ by zero on $\mathbb{R}\setminus[0,T]$, and the right-hand side of \eqref{eq4.8} can be represented as
\begin{align}\label{eq4.9}
   &\sup_{\|\psi\|_{L^{q'(\cdot)}_{t}}\leq 1}\int_{0}^{T}\int_{0}^{T}\frac{\chi_{\{0<s<t\}}|\psi(t)|}{|t-s|^{\frac{1}{\alpha}+\frac{n}{\alpha p}}}dt\|v(s,\cdot)\|_{L^{p}_{x}}^{2}ds\nonumber\\
   &=\sup_{\|\psi\|_{L^{q'(\cdot)}_{t}}\leq 1}\int_{0}^{T}\left(\int_{-\infty}^{+\infty}\frac{|\psi(t)|}{|t-s|^{\frac{1}{\alpha}+\frac{n}{\alpha p}}}dt\right)\|v(s,\cdot)\|_{L^{p}_{x}}^{2}ds\nonumber\\
   &=\sup_{\|\psi\|_{L^{q'(\cdot)}_{t}}\leq 1}\int_{0}^{T}\mathcal{I}_{\beta}(|\psi|)\|v(s,\cdot)\|_{L^{p}_{x}}^{2}ds,
\end{align}
where $\beta=1-\frac{1}{\alpha}-\frac{n}{\alpha p}$. Furthermore, for the right-hand side of \eqref{eq4.9},
by using H\"{o}lder's inequality \eqref{eq2.1} with $\frac{2}{q(x)}+\frac{1}{\widetilde{q}(x)}=1$ yields that
\begin{align}\label{eq4.10}
   \sup_{\|\psi\|_{L^{q'(\cdot)}_{t}}\leq 1}\int_{0}^{T}\mathcal{I}_{\beta}(|\psi|)\|v\|_{L^{p}_{x}}^{2}ds&\leq C\sup_{\|\psi\|_{L^{q'(\cdot)}_{t}}\leq 1}\|\mathcal{I}_{\beta}(|\psi|)\|_{L^{\widetilde{q}(\cdot)}_{t}}\|v\|_{L^{q(\cdot)}_{t}(L^{p}_{x})}^{2}\nonumber\\
   &\leq C\sup_{\|\psi\|_{L^{q'(\cdot)}_{t}}\leq 1}\|\psi\|_{L^{r(\cdot)}_{t}}\|v\|_{L^{q(\cdot)}_{t}(L^{p}_{x})}^{2},
\end{align}
where the above indices satisfy the relationship
$$
\frac{1}{\widetilde{q}(\cdot)}=\frac{1}{r(\cdot)}-(1-\frac{1}{\alpha}-\frac{n}{\alpha p}).
$$
 Since $\frac{1}{\widetilde{q}(x)}=1-\frac{2}{q(x)}$ and $\frac{1}{q'(x)}=1-\frac{1}{q(x)}$, we can deduce that $r(\cdot)<q'(\cdot)$ under the condition $\frac{\alpha}{q(\cdot)}+\frac{n}{p}<\alpha-1$. By using Lemma \ref{le2.3} with $r(\cdot)<q'(\cdot)$ and $\Omega=[0,T]$, we deduce from \eqref{eq4.10} that
\begin{align}\label{eq4.11}
   \sup_{\|\psi\|_{L^{q'(\cdot)}_{t}}\leq 1}\|\psi\|_{L^{r(\cdot)}_{t}}\|v\|_{L^{q(\cdot)}_{t}(L^{p}_{x})}^{2}&\leq \sup_{\|\psi\|_{L^{q'(\cdot)}_{t}}\leq 1}\|\psi\|_{L^{q'(\cdot)}_{t}}\|v\|_{L^{q(\cdot)}_{t}(L^{p}_{x})}^{2}\nonumber\\
   &\leq C(1+T)\|v\|_{L^{q(\cdot)}_{t}(L^{p}_{x})}^{2}.
\end{align}
Finally, taking all above estimates \eqref{eq4.8}--\eqref{eq4.11} into \eqref{eq4.7}, we get \eqref{eq4.6}. Thus, we complete the proof of Lemma \ref{le4.2}.
\end{proof}
\smallbreak

Based on the desired linear and nonlinear estimates obtained in Lemmas \ref{le4.1} and \ref{le4.2}, we know that,  for any $0<T<\infty$,  there exist two positive
constants $C_{1}$ and $C_{2}$ such that
\begin{align}\label{eq4.12}
  \|v\|_{\mathcal{Y}_{T}}\leq
 C_{1}\max\{T^{\frac{1}{q^{-}}}, T^{\frac{1}{q^{+}}}\}\|v_{0}\|_{L^{\bar{p}(\cdot)}_{x}}
  +C_{2}  (1+T)\|v\|_{\mathcal{Y}_{T}}^{2}.
\end{align}
Hence, if we choose $T$ small enough such that
\begin{align*}
\|v_{0}\|_{L^{\bar{p}(\cdot)}_{x}}\leq \frac{1}{4C_{1}C_{2} (1+T)\max\{T^{\frac{1}{q^{-}}}, T^{\frac{1}{q^{+}}}\}},
\end{align*}
we know from Proposition \ref{pro3.1} that the equation \eqref{eq3.2} admits a unique solution $v\in \mathcal{Y}_{T}$, which yields that there exists a unique solution $u$ of the system \eqref{eq1.1} such that $\nabla(-\Delta)^{-1}u\in \mathcal{Y}_{T}$. Thus, we complete the proof of Theorem \ref{th1.4}.
\bigskip

\noindent \textbf{Acknowledgements.}
The authors declared that they have no conflict of interest. G. Vergara-Hermosilla was supported by the ANID postdoctoral program BCH 2022 (no. 74220003), also he thanks Pierre-Gilles Lemari\'{e}-Rieusset and Diego Chamorro for their helpful comments and advises.
J. Zhao was partially supported by the National Natural Science Foundation of China (no. 12361034) and
 the Natural Science Foundation of Shaanxi Province (no. 2022JM-034).


\end{document}